\documentclass[12pt]{amsart}
\usepackage[applemac]{inputenc}

\usepackage{fullpage}
\usepackage{%times,
 ulem, 
 amsfonts}
\usepackage{amsmath,esint, amssymb, amsthm}

\usepackage[colorlinks=true, pdfstartview=FitV, linkcolor=blue, citecolor=blue,
 urlcolor=blue]{hyperref}
\usepackage{color}
\newtheorem{Theo}{Theorem}
 \newtheorem{lem}{Lemma}

 \newtheorem{prop}{Proposition}
 
 \newtheorem{defi}{ Definition}
 \newtheorem*{gracies}{Acknowledgements}
  \newtheorem{remark}{ Remark}
  
\newcommand{\CC}{\mathbb{C}}  
 
\newcommand{\RR}{\mathbb{R}}

\newcommand{\R}{\mathbb{R}}

\author[T. Hmidi]{Taoufik Hmidi}
\address{IRMAR, Universit\'e de Rennes 1\\ Campus de
Beaulieu\\ 35~042 Rennes cedex\\ France}
\email{thmidi@univ-rennes1.fr}

\title[]{On the trivial solutions for the rotating patch model}

\begin{document}
\begin{abstract}
In this paper we study the clockwise simply connected rotating patches for Euler equations. By using the moving plane method we prove that Rankine vortices  are the only solutions to this problem in the class of {\it slightly} convex domains.  We discuss in the second part of the paper the case where the angular velocity $\Omega=\frac12$ and we show without any geometric condition that the set of the V-states  is trivial and reduced to the Rankine vortices.
\end{abstract}
\maketitle
%\tableofcontents
%\maketitle

\section{Introduction}

We shall study in this paper some aspects of the vortex motion for the two-dimensional incompressible  Euler system which can be written with the vorticity-velocity formulation in the form, \begin{equation}\label{vorticity}
%\left\{ \begin{array}{ll}
\begin{cases}
\partial_{t}\omega+v\cdot\nabla  \omega=0, \quad x\in \R^2, \; t > 0,\\
v= \nabla^{\bot}\triangle^{-1}\omega, &
 \\ \omega(0,x) =\omega_{0}(x). &
\end{cases} %\right.
\end{equation}
Here $\nabla^\perp=(-\partial_2,\partial_1),$ $v=(v_1,v_2)$ is the velocity field and the $\omega$ its vorticity given by the
scalar
$\omega=
\partial_1 v_2 -\partial_2 v_1. $ The classical theory dealing with the local$\backslash$global well-posedness of smooth solutions is well developed and we refer for instance to \cite{BM,Ch}.

According to Yudovich result \cite{Y1}  the vorticity equation has a unique
global solution in the weak sense provided the initial vorticity
$\omega_0$ belongs to $L^1 \cap L^\infty$. This result allows to deal rigorously with the so-called vortex patches which are initial  vortices uniformly distributed in a confined region $D$, that is,   $\omega_0=\chi_{D}$  the characteristic function of $D.$ Since the  vorticity is transported along
trajectories, we conclude that the vorticity preserves the vortex patch structure for any positive time. This means that for any $t\geq0$,     $\omega(t)=\chi_{D_t}$, with  $D_t=\psi(t,D)$ is the image
of $D$ by the flow $\psi$ which satisfies the ordinary differential equation
\begin{equation}\label{flux}
\partial_t\psi(t,x)= v(t,\psi(t,x)), \quad \psi(0,x)=x.
\end{equation}
The dynamics of the boundary of $D_t$ is in general  complex and very difficult to follow. By using the contour dynamics method we may parametrize the boundary by a function $\gamma_t:\mathbb{T}\to \partial D_t$ satisfying   a nonlinear and non local equation of the following type
$$
\partial_t\gamma_t=-\frac{1}{2\pi}\int_{\partial D_t}\log|\gamma_t-\xi|d\xi.
$$
There are  few examples known in the literature  with explicit   dynamics. The first one is Rankine vortex where $D$ is a disc, in this case the particle trajectories are
circles centered at the origin, and  therefore $D_t =D, \;\forall\, t\geq0.$ The second example   is a remarkable one and  discovered by Kirchhoff \cite{Kirc} is the ellipses. In this case the domain $D_t$ does not change its shape and  undergoes a perpetual rotation around its barycenter  with uniform angular velocity  $\Omega$ 
related to the semi-axes $a$ and $b$ 
through the formula $\Omega = ab/(a+b)^2.$ See, for instance,
\cite[p. 304]{BM}.

It seems that the ellipses are till now the only explicit example with such properties but wether or not  other non trivial implicit  rotating patches  exist has been discussed in the last few decades from numerical and theoretical point of view. To be more precise about these structures, we say that $\omega_0=\chi_D$ is a V-state or a rotating patch if there exists a real \mbox{number $\Omega$} called the angular velocity such that the  support of the vorticity $\omega(t)=\chi_{D_t}$ is described by
$$
D_t={R}_{x_0,\Omega t} D,\quad \forall t\geq 0,
$$ 
with ${R}_{x_0,\Omega t}$ being the planar rotation with center $x_0$ and angle $\Omega t.$ Deem and Zabusky \cite{DZ} were the first to reveal numerically  the existence of 
simply connected V-states with the {$m-$fold} symmetry for the integer $m=3,4,5.$ Recall that  a domain $D$ is said to be  $m$-fold symmetric if { it is invariant by the dihedral group $D_m$ which is the symmetry group  of a regular polygon of $m$ sides}.  A few years later, Burbea \cite{B} gave an analytic proof by using the 
bifurcation theory showing the existence of a countable family of V-states with the $m$-fold symmetry for any $m\geq2$. They can be identified to one-dimensional branches bifurcating from the Rankine vortex at the  simple ``eigenvalues" $\big\{\Omega=\frac{m-1}{2m}, m\geq2\big\}.$ See also \cite{HMV}, where the
$C^\infty$ boundary regularity of the bifurcated $V-$states close to the
disc was proven. It seems that close to the disc, the bifurcating branches rotate with  bounded angular velocities, $\Omega\in ]0,\frac12[$. It is important to know wether all the V-states possess an angular velocity in this strip. From the Implicit Function Theorem we know that close to the disc there are no non trivial V-states associated to  $\Omega\notin [0,\frac12]$. 

 In this paper we give a partial answer to this problem. We shall first show that there is no clockwise rotating patches, that is, $\Omega\leq 0,$ but with some geometric constraints. When $\Omega=0$, this corresponds to stationary patches and we know from a recent result of Fraenkel \cite{Fran} in gravitational theory that the discs are the only stationary patches. He used the techniques of moving plane method which can be adapted to our framework   only when   $\Omega\le0.$ More precisely, we obtain the following result.
\begin{Theo}\label{thm1}
Let $D$ be a $C^1$ bounded  simply connected  domain convex or more generally being in the class  $\Sigma_{\arccos\frac{1}{\sqrt{5}}}$ introduced in the Definition $\ref{def2}$. Assume that $\chi_D$ is a  V-state satisfying the equation \eqref{vorticity} with the angular velocity $\Omega<0$. Then necessarily $D$ is a disc.
\end{Theo}
The proof uses the moving plane method  in the spirit of the papers \cite{Fran, wolfgang}. We start first with reformulating the equation in an integral form by using the strong maximum principle. This can be done by noticing that  the stream function $\psi$ associated to the vorticity $\chi_D$  is a V-state if and only if  the following  equation holds true
\begin{equation}\label{Eqa1}
\varphi\triangleq\mu+\frac12\Omega |x|^2-\psi(x)=0,\quad\forall\,  x\in \partial D, 
\end{equation}
with $\mu$ a constant.
From the maximum principle applied to $\varphi$ in the domain $D$ one deduces that $D\subset\{\varphi >0\}.$ To get the equality between the latter sets which is essential for our approach we need to prove that $\varphi(x)<0$ for any $x\notin \overline{D}$. It is not clear how to get this result without any geometric constraint on the domain because the  function $\varphi$ is  only superharmonic in the complement of $D$ and  going to $-\infty$ at infinity and  therefore the associated  minimum principle is not conclusive. Thus we should put some geometric constraints on the domain in order to get the desired result. Notice that this kind of problem does not appear when $\Omega=0$ because  the function $\varphi$ is harmonic and the maximum principle can be applied.  Now once the domain is fully described by $\varphi$ the V-state equation can be reduced to a nonlinear integral equation on $\varphi$, 
\begin{equation*}
\varphi(x)=\mu+\frac12\Omega|x|^2-\frac{1}{2\pi}\int_{\RR^2}\log|x-y|\, H(\varphi(y))\,dy,\quad \forall x\in \RR^2.
\end{equation*}
with $H=\chi_{[0,+\infty[}$ being the Heaviside function. The second step consists in applying the techniques of the moving plane method in order to show that any solution $\varphi$ is radial and strictly decreasing.  Therefore the conclusion follows from the fact that  the boundary $\partial D$ is a level set of $\varphi$ and it must be a circle.

When $\Omega\geq0$, there is a competition between the quadratic potential describing the rotation effect and the gravitational potential. This competition is fruitful  in the strip  $\Omega\in ]0,\frac12[$ and  leads to non trivial examples as we have previously quoted concerning the existence of the $m$-folds.  The approach of moving plane method fails for $\Omega\geq\frac12$ because we are still in the region where there is an active  competition between the two potentials and we  lose the monotonicity of the nonlinear functional. Our second purpose in this paper is to establish    a global result for  the end-point  $\Omega=\frac12$ which is  very special. We shall prove a similar result to Theorem \ref{thm1} by using different techniques from  complex analysis.   Our result reads as follows.
\begin{Theo}\label{thm2}
Let $D$ be a $C^1$ bounded  simply connected  domain. Assume that $\chi_D$ is a  V-state satisfying the equation \eqref{vorticity} with the angular velocity $\Omega=\frac12$. Then necessarily $D$ is a disc.
\end{Theo}
According to  the definition  \eqref{Eqa1} we can see that the value $\Omega=\frac12$ is very special because in this case $\varphi$ is harmonic inside the domain $D$ and vanishes on its boundary. This leads to an explicit formula for the Newtonian potential in $D$ which is quadratic 
%in itdealing with the complex formulation  of the rotating patches, the value $\Omega=\frac12$ is very particular; there will be   an algebraic cancellation allowing to kill the complex singularities coming from the rotation by the singular part of the Newtonian potential. We can afterwards use the maximum principle for harmonic functions and give an explicit expression for the Cauchy transform of the domain which 
and  turns out to be that of a circle. This is a kind of inverse problems was solved in a more general framework   in \cite{HMV}. The proof in our particular situation is more easier and for the convenience of the reader we shall give a complete proof of this fact. 
\begin{remark}
In Theorem $\ref{thm2}$ we do not assume any geometric property for the V-states and  therefore it is interesting to see whether we can get  rid of the geometric constraints in Theorem $\ref{thm1}.$ In the last section we shall deal with the case $\Omega=\frac12$ and give a proof for Theorem \ref{thm2}.
\end{remark}
\begin{remark}
In our results we require the boundary to be at least of class $C^1$. This appears when we move from the equation of the vorticity to the stream function formulation. The regularity could be relaxed to rectifiable boundaries.
\end{remark}
%\begin{remark}
%The doubly connected case seems to be not automatically covered by our approach. We can try to implement the techniques of the moving planes but the Heaviside function appearing in the simply connected case  is replaced by a combination of   two Heaviside  functions with opposite signs. Therefore we loose the monotonicity property  which is very crucial to conclude that the solution is radial. Thus whether or not clockwise doubly connected V-states  exist remains open.
%\end{remark}
\begin{remark}
Wether  there are   V-states rotating faster than $\frac12$ remains open.
\end{remark}
The paper is organized as follows. In Section $2$ we shall discuss the mathematical model of the V-states. Section $3$ is devoted to the clockwise rotating patches and where we prove Theorem \ref{thm1}.

\section{Model}
Recall that a rotating patch, called  also  V-state, for Euler equations written in the \mbox{form \eqref{vorticity}} is a solution of the type $\omega(t)=\chi_{D_t}$ where the domain $D_t$ rotates uniformly with an angle velocity $\Omega$ around its barycenter  assumed to be zero, that is,
$$
D_t=R_{0,\Omega t}D,
$$
with $R_{0,\Omega t}$ being the rotation of center $0$ and angle $\Omega t.$
According to \cite{B,HMV} when the boundary of the patch is smooth this holds true if and only if 
\begin{equation}\label{master0}
\big( v(x)-\Omega x^\perp\big)\cdot \vec{n}(x)=0,\quad\forall x\in \partial D,
\end{equation}
with $\vec{n}(x)$ be the  outward-pointing normal unit vector to the boundary and   $v$ be the induced velocity by the patch $\chi_D$ which can be recovered from Biot-Savart law in its complex form as follows,
$$
v(z)=\frac{i}{2\pi}\int_{D}\frac{1}{\overline{z}-\overline{y}}dA(y), \quad z\in  \CC.
$$
 Integrating the equation \eqref{master0} yields
\begin{equation}\label{master}
\psi(z)-\frac12\Omega|z|^2=Cte\triangleq\mu,\quad\forall z\in \partial D,
\end{equation}
with $\psi$ the stream function associated to the patch $\chi_{D}$ and defined by
$$
\psi(z)=\frac{1}{2\pi}\int_{D}\log|z-y|\, dA(y),\quad z\in \CC.
$$
Recall the identity
\begin{eqnarray*}
\partial_z\psi(z)&=&\frac12 i\overline{v(z)}\\
&=&\frac{1}{4\pi}\int_{D}\frac{1}{{z}-{y}}dA(y).
\end{eqnarray*}
From the  Cauchy-Pompeiu formula one can write
   \begin{equation*}
 \frac{1}{2\pi i}\int_{\Gamma}\frac{\overline{z}-\overline{\xi}}{\xi
 -z}d\xi=\frac{1}{\pi}\int_{D}\frac{1}{{z}-{y}}dA(y), \quad z \in \CC.
    \end{equation*}
    Consequently
     \begin{equation}\label{GC1}
4 \partial_z\psi(z)=\frac{1}{2\pi i}\int_{\Gamma}\frac{\overline{z}-\overline{\xi}}{\xi
 -z}d\xi
, \quad \forall  z \in \CC.
    \end{equation}
    Thus  the equation \ref{master0} can be written in the following   complex form
 $$
 {\it{Re}}\Big\{\Big(2\Omega\overline{z}+\frac{1}{2i\pi}\int_{\partial D}\frac{\overline{\xi}-\overline{z}}{\xi-z}d\xi\Big)\vec{\tau}(z)\Big\}=0,\quad\forall\, z\in \partial D,
 $$
with $\vec{\tau}(z)$ being a tangent unit  vector to the boundary at the point $z$.  
%We shall transform it into another form.
% We introduce the Cauchy integral of $\overline{z}$ on $D,$ defined by
%$$
%\gamma(z)=\frac{1}{2i\pi}\int_{\partial D}\frac{\overline{\xi}}{\xi-z}d\xi, z\in D.
%$$
%It  is obviously that $\gamma$ is holomorphic in $D$ and we can show that  for $z\in \partial D$  the limit from inside the domain $D$ of this function exists and will be  denoted by $\gamma(z).$
%It is plain to see that 
%$$
%\frac{1}{2i\pi}\int_{\partial D}\frac{\overline{\xi}-\overline{z}}{\xi-z}d\xi=\gamma(z)-\overline{z}, \quad z\in \overline{D}.
%$$
%Thus the equation of the V-states becomes
% \begin{equation}\label{pilot}
% {\it{Re}}\Big\{\Big((2\Omega-1)\overline{z}+\gamma(z)\Big)\vec{\tau}(z)\Big\}=0,\quad z\in \partial D.
% \end{equation}
%Integrating this equation we find
% \begin{equation}\label{pilot1}
%(\Omega-\frac12)|z|^2+{\it{Re}}\Big( \frac{1}{2i\pi}\int_{\partial D}\overline{\xi}\log\big(1-\frac{z}{\xi}\big)d\xi\Big)=Cte, z\in \partial D
% \end{equation}
% For more details about the derivation of this equation see \cite{B,HMV}.
% Notice that the function $F$ defined below is holomorphic in $D$ and admits a continuous extension to the boundary,
% $$
% F(z)=\frac{1}{2i\pi}\int_{\partial D}\overline{\xi}\log\big(1-\frac{z}{\xi}\big)d\xi,\quad z\in \overline{D}.
% $$
 \section{Clockwise rotating patches}
Now we shall prove the main result stated in Theorem \ref{thm1}.  The case $\Omega=0$ was done in \cite{Fran} and therefore we shall focus only on  $\Omega<0.$ It  will be done in several steps in the spirit of the paper \cite{wolfgang}. Firstly, we shall transform the V-states equation into an integral one. To do so we need some strong geometric conditions on the domains. Secondly, we use the moving plane method to prove that all the solutions of the integral equation are radial symmetric and strictly decreasing. From this result we can deduce easily Theorem \ref{thm1}. 

\subsection{Integral equation}
Our first  goal is to transform the V-states equation into an integral form. This will be done under some geometric constraints on the domains of the  V-states. We shall for this purpose need some definitions.
\begin{defi}\label{defaz1}
Let $D$ be a $C^1$ simply connected domain. Let $x_0\in \partial D$, $\alpha\in [0,\frac\pi2]$ and  define the sector
$$
\hbox{Sec}_{x_0,\alpha}=\Big\{x\in\RR^2\backslash\{x_0\};\, \cos\alpha\leq \frac{x-x_0}{\|x-x_0\|}\cdot \vec{\nu}(x_0) \Big\}
$$
with $\vec{\nu}(x_0)$ the outward-pointing normal unit vector to the boundary at $x_0.$
We also define the  sets
$$
D_{x_0}^{+}=\Big\{x\in D;\,(x-x_0)\cdot \vec{\nu}(x_0)\geq0\Big\} 
$$
and 
$$
D_{x_0}^{-}=\Big\{y=x-2\big[(x-x_0)\cdot \vec{\nu}(x_0)\big] \vec{\nu}(x_0);\, x\in D_{x_0}^+\Big\}. 
$$

\end{defi}
Note that the set $D_{x_0}^+$ is the part of $D$ located above  the tangent line to the boundary at the point $x_0$ (the orientation is with respect to $\vec{\nu}$). As to  the set $D_{x_0}^{-}$, it  is the reflection of the set $D_{x_0}^{+}$ with respect to this tangent.
\begin{defi}\label{def2}
Let $D$ be a $C^1$ simply connected bounded domain  with zero as barycenter.
 We say that $D$  belongs to the class $\Sigma_\alpha$ with $\alpha\in [0,\frac\pi2]$ if
\begin{enumerate}
\item For each $x_0\in \partial D$, we have $x_0\cdot \vec{\nu}(x_0)\geq0$.
\item For each $x_0\in \partial D$, we have $\hbox{Sec}_{x_0,\alpha}\cap  D=\varnothing$. 
\item For each  $x_0\in \partial D$, the subset $D_{x_0}^{-}$ introduced in the Definition $\ref{defaz1}$ is contained in the domain $D.$ 

\end{enumerate}

\end{defi}
We shall make some comments.
\begin{remark}
\begin{enumerate}
\item  The first assumption in the previous definition means that the barycenter is always below any tangent line to the boundary.
\item If a domain belongs to the class $\Sigma_\alpha$ then it  belongs to the class $\Sigma_{\alpha^\prime}$ for any $\alpha^\prime\in[0,\alpha].$
\item Any convex domain belongs to $\Sigma_{\frac\pi2}$.
\item  Roughly speaking, when $\alpha$ is close to $\frac\pi2$ a domain in $\Sigma_\alpha$ is  "slightly convex".
\end{enumerate}

\end{remark}
Now we shall write down an integral equation for $\varphi.$
\begin{prop}\label{prop04}
Let  $D$ be  a $C^1$ simply connected domain belonging to  the class $\Sigma_{\arccos\frac{1}{\sqrt{5}}}$ introduced in the Definition $\ref{def2}$. Assume that $\chi_D$ is a V-state rotating with a negative angular velocity $\Omega<0$. Let  $\psi$ be  the stream function of $\chi_D$, then there exists a constant $\mu$  previously defined in \eqref{master} such that
\begin{equation}\label{master1}
\varphi(x)=\mu+\frac12\Omega|x|^2-\frac{1}{2\pi}\int_{\RR^2}\log|x-y|\, H(\varphi(y))\,dy,\quad \forall x\in \RR^2.
\end{equation}
with $H=\chi_{[0,+\infty[}$ being  the Heaviside function and 
$$
\varphi(x)\triangleq\mu+\frac12\Omega|x|^2-\psi(x).
$$
\end{prop}
The proof is an immediate consequence of the next lemma which allows to write 
$$
\psi(x)=\frac{1}{2\pi}\int_{\RR^2}\log|x-y|\, H(\varphi(y))\,dy.
$$
The lemma reads as follows.
\begin{lem}\label{lem1}
Let  $\Omega<0$ and $D$ be a domain as in the Proposition $\ref{prop04}$, then
$$
D=\big\{x\in\RR^2, \varphi(x)>0\big\},\quad \partial D=\big\{x\in\RR^2, \varphi(x)=0\big\}.
$$
In particular
$$
\forall x\in \RR^2,\quad \chi_D(x)=H(\varphi(x)).
$$

\end{lem}

\begin{proof}
The proof will be done in two steps. In the first one we show that $D\subset \big\{x\in\RR^2, \varphi(x)>0\big\}$ which does not require any geometric constraint on the domain and we use only  the maximum principle applied to $\varphi$.  However the converse inclusion is more subtle and it is not clear wether it can be proven without any additional geometric properties for the domain. To get the first inclusion  we observe that  $\varphi$ satisfies  in the distribution sense
\begin{equation}\label{T1}
\Delta \varphi=2\Omega-\chi_D\le0,\quad \hbox{in}\quad \RR^2.
\end{equation}
Taking the restriction of this equation to the domain $D$ we get in the classical sense the Dirichlet problem
\begin{equation}\label{B}
\left\{ \begin{array}{ll}
\Delta \varphi<0\quad \hbox{in}\quad D&\\
\varphi=0,\quad \partial D.
\end{array} \right. 
\end{equation} 
Applying the strong maximum principle we find 
$$
\forall x\in D, \quad \varphi(x)>0
$$
and therefore we obtain the inclusion
\begin{equation}\label{inc1}
D\subset\big\{x\in \RR^2, \varphi(x)>0\big\}.
\end{equation}
We should now prove the converse which follows from 
\begin{equation}\label{incc1}
\RR^2\backslash \overline{D}\subset\big\{x\in \RR^2, \varphi(x)<0\big\}.
\end{equation}
To get this inclusion  we shall first give an elementary proof when the domain $D$ is assumed to be convex which is in fact belongs to the class $\Sigma_\alpha$ and come back later to the general case.
Since $\psi$ has a logarithmic growth at infinity and $\Omega<0$ then
$$
\lim_{|x|\to+\infty}\varphi(x)=-\infty.
$$
To get the inclusion it suffices to show that $\varphi$ has no critical points outside $\overline{D}$. Assume that there exits $x_0\notin\overline{D}$ such that $\nabla\varphi(x_0)=0.$ It is easy to see that
\begin{eqnarray*}
\nabla\varphi(x_0)&=&\Omega x_0-\frac{1}{2\pi}\int_{D}\frac{x_0-y}{|x_0-y|^2}dy\\
&=&0.
\end{eqnarray*}
As ${D}$ is convex one can find, using the separation theorem, a unit vector $e$ (for example the outward-pointing normal unit vector) such that 
$$
(x_0-y)\cdot e>0,\quad \forall y\in D.
$$
Consequently we get
\begin{eqnarray*}
\nabla\varphi(x_0)\cdot e&=&\Omega\, x_0\cdot e-\frac{1}{2\pi}\int_{D}\frac{(x_0-y)\cdot e}{|x_0-y|^2}dy\\
&<&\Omega\, x_0\cdot e\\
&<&0,
\end{eqnarray*}
where we have used in the last line the fact that $0\in D$ and  $x_0\cdot e>0.$ This gives the desired result in the case of convex domains. Now let us discuss the general case of domains belonging to the class $\Sigma_{\arccos\frac{1}{\sqrt{5}}}$. We claim that for any $x_0\in \partial D$ the function $g:[0,+\infty)\to \RR$  defined by
\begin{eqnarray*}
g(t)&\triangleq&\varphi\big(x_0+t\vec{\nu}(x_0)\big)\\
&=&\mu+ \frac12\Omega|x_0+t\vec{\nu}(x_0)|^2-\psi\big(x_0+t\vec{\nu}(x_0)\big)
\end{eqnarray*}
is strictly decreasing. Indeed, by differentiation we find
$$
g^\prime(t)=\Omega(t+x_0\cdot\vec{\nu}(x_0))-\nabla\psi\big(x_0+t\vec{\nu}(x_0)\big)\cdot\vec{\nu}(x_0).
$$
Using the assumption $(1)$ of the  Definition \ref{def2} and the fact that $\Omega<0$ we obtain
\begin{equation}\label{ineq1}
\forall t\geq0,\quad g^\prime(t)< -\nabla\psi\big(x_0+t\vec{\nu}(x_0)\big)\cdot\vec{\nu}(x_0).
\end{equation}
By straightforward computations we get
\begin{eqnarray}\label{ineq45}
\nonumber \nabla\psi\big(x_0+t\vec{\nu}(x_0)\big)\cdot\vec{\nu}(x_0)&=&\frac{1}{2\pi}\int_{D}\frac{\big(x_0+t\vec{\nu}(x_0)-y\big)\cdot\vec{\nu}(x_0)}{|x_0+t\vec{\nu}(x_0)-y|^2}dy\\
&=&\frac{t}{2\pi}\int_{D}\frac{1}{|x_0+t\vec{\nu}(x_0)-y|^2}dy+\frac{1}{2\pi}\int_{D}\frac{\big(x_0-y\big)\cdot\vec{\nu}(x_0)}{|x_0+t\vec{\nu}(x_0)-y|^2}dy.
\end{eqnarray}
From  the Definition \ref{def2} we get the partition $D=D_{x_0}^{-}\cup D_{x_0}^{+}\cup \hat{D}_{x_0}$ with $\hat{D}_{x_0}$ the complement of $D_{x_0}^{-}\cup D_{x_0}^{+}$ in $D.$ Therefore the last integral term may be written in the form
\begin{eqnarray}\label{ineq46}
\nonumber \int_{D}\frac{\big(x_0-y\big)\cdot\vec{\nu}(x_0)}{|x_0+t\vec{\nu}(x_0)-y|^2}dy&=&\int_{\hat{D}_{x_0}}\frac{\big(x_0-y\big)\cdot\vec{\nu}(x_0)}{|x_0+t\vec{\nu}(x_0)-y|^2}dy+\int_{D_{x_0}^{-}\cup D_{x_0}^{+}}\frac{\big(x_0-y\big)\cdot\vec{\nu}(x_0)}{|x_0+t\vec{\nu}(x_0)-y|^2}dy\\
&\geq&\int_{D_{x_0}^{-}\cup D_{x_0}^{+}}\frac{\big(x_0-y\big)\cdot\vec{\nu}(x_0)}{|x_0+t\vec{\nu}(x_0)-y|^2}dy\triangleq I_{x_0},
\end{eqnarray}
where we have used the fact that $\big(x_0-y\big)\cdot\vec{\nu}(x_0)\geq0$ for any $y\in \hat{D}_{x_0}$. We shall use the change of variables 
$$z\in D_{x_0}^{-}\mapsto y=z-2\big[(z-x_0)\cdot\vec{\nu}(x_0)\big]\vec{\nu}(x_0)\in D_{x_0}^{+}
$$ which is a diffeomorphism preserving Lebesgue measure and therefore,
\begin{eqnarray*}
\int_{D_{x_0}^{+}}\frac{\big(x_0-y\big)\cdot\vec{\nu}(x_0)}{|x_0+t\vec{\nu}(x_0)-y|^2}dy&=&-\int_{D_{x_0}^{-}}\frac{\big(x_0-z\big)\cdot\vec{\nu}(x_0)}{|x_0-t\vec{\nu}(x_0)-z|^2}dz.
\end{eqnarray*}
Consequently,
\begin{eqnarray*}
I_{x_0}&=&\int_{ D_{x_0}^{-}}{\big(x_0-y\big)\cdot\vec{\nu}(x_0)}\big(\frac{1}{|x_0+t\vec{\nu}(x_0)-y|^2}-\frac{1}{|x_0-t\vec{\nu}(x_0)-y|^2}\big)dy\\
&=&-4t\int_{ D_{x_0}^{-}}\frac{\big[\big(x_0-y\big)\cdot\vec{\nu}(x_0)\big]^2}{|x_0+t\vec{\nu}(x_0)-y|^2|x_0-t\vec{\nu}(x_0)-y|^2}dy.%\\
%&\geq&-4t\int_{ D_{x_0}^{+}}{\big[\big(x_0-y\big)\cdot\vec{\nu}(x_0)}\big]^2\frac{1}{|x_0+t\vec{\nu}(x_0)-y|^4}dy\
\end{eqnarray*}
Putting together the preceding identity with \eqref{ineq46} and \eqref{ineq45} we obtain
\begin{eqnarray*}
\nabla\psi\big(x_0+t\vec{\nu}(x_0)\big)\cdot\vec{\nu}(x_0)&\geq& \frac{t}{2\pi}\int_{ D_{x_0}^{-}}\frac{1}{|x_0+t\vec{\nu}(x_0)-y|^2}\Big(1-4\frac{\big[\big(x_0-y\big)\cdot\vec{\nu}(x_0)\big]^2}{|x_0-t\vec{\nu}(x_0)-y|^2}\Big)dy.%\\
%&\geq&\frac{t}{2\pi}\int_{ D_{x_0}^{+}}\frac{1}{|x_0+t\vec{\nu}(x_0)-y|^2}\Big(1-4\frac{\big[\big(x_0-y\big)\cdot\vec{\nu}(x_0)\big]^2}{|x_0-y|^2}\Big)dy
\end{eqnarray*}
Easy computations show that
$$
\inf_{t\geq0}|x_0-t\vec{\nu}(x_0)-y|^2=|x_0-y|^2-\big[(x_0-y)\cdot\vec{\nu}(x_0)\big]^2
$$
and therefore
\begin{eqnarray*}
\nabla\psi\big(x_0+t\vec{\nu}(x_0)\big)\cdot\vec{\nu}(x_0)&\geq& \frac{t}{2\pi}\int_{ D_{x_0}^{-}}\frac{|x_0-y|^2}{|x_0+t\vec{\nu}(x_0)-y|^2}\Big(\frac{1-5{\big[\frac{x_0-y}{|x_0-y|}\cdot\vec{\nu}(x_0)\big]^2}}{|x_0-y|^2-\big[(x_0-y)\cdot\vec{\nu}(x_0)\big]^2}\Big)dy.%\\
%&\geq&\frac{t}{2\pi}\int_{ D_{x_0}^{+}}\frac{1}{|x_0+t\vec{\nu}(x_0)-y|^2}\Big(1-4\frac{\big[\big(x_0-y\big)\cdot\vec{\nu}(x_0)\big]^2}{|x_0-y|^2}\Big)dy
\end{eqnarray*}
Recall that for $y\in D_{x_0}^-$
$$
\frac{x_0-y}{|x_0-y|}\cdot\vec{\nu}(x_0)=\frac{z-x_0}{|z-x_0|}\cdot\vec{\nu}(x_0),\quad{\rm{with}}\quad z= y-2\big[(y-x_0)\cdot\vec{\nu}(x_0)\big]\vec{\nu}(x_0)\in D_{x_0}^+.
$$
According to the definition of $D_{x_0}^+$ and  the condition  $(2)$ of the Definition \ref{def2} one sees that
$$
0\leq\frac{(z-x_0)}{|x_0-z|}\cdot\vec{\nu}(x_0)\le \frac{1}{\sqrt{5}}
$$
and therefore 
$$
\forall y\in D_{x_0}^{-},\quad 0\leq\frac{(x_0-y)}{|x_0-y|}\cdot\vec{\nu}(x_0)\le \frac{1}{\sqrt{5}}\cdot
$$
This  implies that
$$
\forall t\geq0,\quad \nabla\psi\big(x_0+t\vec{\nu}(x_0)\big)\cdot\vec{\nu}(x_0)\geq0.
$$
Combining this inequality with \eqref{ineq1} one gets
$$
\forall t\geq 0,\quad g^\prime(t)<0.
$$
It follows that $g$ is strictly decreasing and 
$$
\forall t>0,\quad g(t)<g(0)=0.
$$
In the last equality we have used the equation of the V-states, that is, $\varphi(x)=0,\, \forall \,x\in \partial D$.
Coming back to the definition of $g$ this proves that
$$
\forall t>0,\quad \varphi(x_0+t\vec{\nu}(x_0))<0.
$$
According to the condition $(2)$ of the Definition \ref{def2} we get
$$
\RR^2\backslash \overline{D}=\cup_{t>0, x_0\in\partial D}]x_0,x_0+t\vec{\nu}(x_0)].
$$
This allows to conclude that 
$$
\forall x\in \RR^2\backslash \overline{D},\quad \, \varphi(x)<0
$$
and therefore we get the desired inclusion \eqref{incc1}. The proof of the lemma is now complete.
%Take $R$ large enough such that $\sup_{|x|=R}\varphi(x)<0$ then
%\begin{equation}\label{B1}
%\left\{ \begin{array}{ll}
%\Delta \varphi<0\quad \hbox{in}\quad D_R&\\
%\varphi\leq0,\quad \partial D_R.
%\end{array} \right. 
%\end{equation} 
%Using once again the strong  maximum principle we get
%$$
%\forall x\in D_R,\quad \varphi(x)<0.
%$$
%This is true for any sufficiently large $R$ and consequently by letting $R$ go to infinity we get
%$$
%\forall x\in \RR^2\backslash \overline{D},\quad \varphi(x)<0
%$$
%which can be written in the form 
%$$
%\RR^2\backslash \overline{D}\subset\big\{x\in \RR^2, \varphi(x)<0\big\}.
%$$
%Together with \eqref{inc1} we find  the desired result.
\end{proof}
\subsection{Moving plane method}
We shall now discuss the symmetry property of any solution of the integral equation \eqref{master1}. This be done by using the moving plane method in the spirit of the papers \cite{Fran,wolfgang}. Our result reads as follows.
\begin{prop}\label{prod1}
Let $\Omega<0$ and $\varphi$ be a solution of \eqref{master1}. Then $\varphi$ is radial and strictly decreasing with respect to the radial variable $r.$
\end{prop}
The result of Theorem \ref{thm1} follows easily from this proposition because the level sets of $\varphi$ are circles centered at zero and   $\partial D\subset \varphi^{-1}(\{0\})$.
As we can see Proposition \ref{prod1} is nothing but the second part $(2)$ of Proposition \ref{prop4}. \\
Next we shall establish a list of  auxiliary results  needed  later in the proof of \mbox{Proposition \ref{prop4}.} \\
{\it Notation:} For $\lambda>0$ we introduce the sets
$$
H_\lambda=\Big\{x=(x_1,x_2)\in\RR^2, x_1<\lambda\Big\}, \quad T_\lambda=\Big\{x=(\lambda,x_2)\in\RR^2, x_2\in \RR\Big\}
$$
and we define for any $x=(x_1,x_2)\in \RR^2$
$$\quad x_\lambda\triangleq(2\lambda-x_1,x_2),\quad \varphi_\lambda(x)\triangleq\varphi(x)-\varphi(x_\lambda).
$$
Our main goal now is to  prove the following result.
\begin{prop}\label{prop2}
There exists $\lambda^\star>0$ such that for any $\lambda\geq\lambda^\star$,
$$
\varphi_\lambda(x)>0,\quad\forall x\in H_\lambda.
$$
In addition,  $\forall \,\lambda_0>0$ there exists $R(\lambda_0)>0$ such that $\forall\, \lambda\geq\lambda_0$
$$
 \forall \, x\in H_\lambda, \quad |x|\geq R(\lambda_0)\Longrightarrow \varphi_\lambda(x)>0.
$$
\end{prop}
\begin{proof}
Using Lemma  $6$ of \cite{wolfgang} together with the fact that  $0$ is the barycenter of $D$  one may write 
$$
\forall x\in \RR^2\backslash\{0\},\quad \varphi(x)=\mu+\frac12\Omega|x|^2-\frac{1}{2\pi}|D|\log|x|+h(x),
$$
with $h$ being a  smooth function outside $D$ and satisfies the asymptotic behavior: there exists $R>0$ large enough such that
$$
\forall |x|\geq R,\quad | h(x)|\le C|x|^{-2};\quad |\nabla h(x)|\le C|x|^{-3}.
$$
From   the elementary fact $|x_\lambda|>|x|$ for $x\in H_\lambda$ we shall get  
\begin{eqnarray*}
\varphi_\lambda(x)&=&\frac12\Omega (|x|^2-|x_\lambda|^2)+\frac{1}{2\pi}|D|\log\big({|x_\lambda|}/{|x|}\big)+h(x)-h(x_\lambda)\\
&> &-2\lambda\Omega(\lambda-x_1)+h(x)-h(x_\lambda).
\end{eqnarray*}
Using the mean value theorem we obtain  for $|x|\geq R,$
\begin{eqnarray*}
|h(x)-h(x_\lambda)|&\le&|x-x_\lambda|\sup_{y\in [x,x_\lambda]}|\nabla h(y)|\\
&\le& C(\lambda-x_1)|x|^{-3}\\
&\le& C(\lambda-x_1) R^{-3}.
\end{eqnarray*}
Hence we get for  $x\in H_\lambda$ with $|x|\geq R$
\begin{eqnarray*}
\varphi_\lambda(x)
&>&(\lambda-x_1)\big(-C R^{-3}-2\lambda\Omega\big).
\end{eqnarray*}
We shall now  see how to  deduce from this inequality the second claim of the proposition and the first claim is postponed later to the end of the proof. Let $\lambda_0>0$ and $\lambda\geq\lambda_0$ and take $x\in H_\lambda$ with $|x|\geq R$. Then it is plain that,
\begin{eqnarray*}
\varphi_\lambda(x)
&>&(\lambda-x_1)\big(-C R^{-3}-2\lambda_0\Omega\big).
\end{eqnarray*}
Now we choose  $R$ such that 
$$
\lambda_0=-\frac{C}{2\Omega} R^{-3}
$$
which  guarantees that $\varphi_\lambda(x)>0$ and this concludes the proof  of the  second part of Proposition \ref{prop2}. Now let us come back to the proof  of the first claim of the proposition. For this aim, it suffices to check  in the preceding claim that the inequality   
$
\varphi_\lambda(x)>0
$
remains true for any $x\in H_\lambda$ provided that $\lambda$ is sufficiently large.
Now let $|x|\le R$ and set
$$
M\triangleq \sup_{|x|\le  R}|\varphi(x)|.
$$
From the asymptotic behavior of $\varphi$ we see that
$$
\lim_{|x|\to\infty}\varphi(x)=-\infty.
$$
In particular there exists $A>0$ such that 
$$
|x|\geq A\Longrightarrow \varphi(x)<-2M.
$$
It is easy to find $\lambda_1$ depending only on $R$ and $A$ such that
\begin{eqnarray*}
\forall \lambda\geq \lambda_1,\quad \forall |x|\le R&\Longrightarrow& |x_\lambda|\geq A\\
&\Longrightarrow&\varphi_\lambda(x)>M.
\end{eqnarray*}
Set $\lambda^\star=\max\{\lambda_0,\lambda_1\}$ then for any $\lambda\geq\lambda^\star$ we obtain
$$
\forall x\in H_\lambda,\quad \varphi_\lambda(x)>0.
$$
This concludes the proof of  Proposition \ref{prop2}.
\end{proof}
The next result deals with a continuation principle very useful to prove the strictly monotonicity of $\varphi.$
\begin{prop}\label{prop3}
Let $\lambda>0$ and assume that $\varphi_\lambda\geq0$ in $H_\lambda$. Then
$$
\varphi_\lambda>0\quad\hbox{in}\quad H_\lambda\quad \hbox{and}\quad \partial_{x_1}\varphi_\lambda(x)<0,\, \forall\, x\in T_\lambda.
$$
\end{prop}
\begin{proof}
First recall that the stream function can be written in the form
$$
\psi(x)=\frac{1}{2\pi}\int_{\RR^2}\log|x-y|\, H(\varphi(y))\,dy.
$$
By using the change of variables $y\mapsto y_\lambda$  which preserves Lebesgue measure one gets
\begin{eqnarray*}
\psi(x)&=&\frac{1}{2\pi}\int_{H_\lambda}\log|x-y|\, H(\varphi(y))\,dy+\frac{1}{2\pi}\int_{H_\lambda^c}\log|x-y|\, H(\varphi(y))\,dy\\
&=&\frac{1}{2\pi}\int_{H_\lambda}\log|x-y|\, H(\varphi(y))\,dy+\frac{1}{2\pi}\int_{H_\lambda}\log|x-y_\lambda|\, H(\varphi(y_\lambda))\,dy.
\end{eqnarray*}
Since the reflection with respect to $T_\lambda$ is an isometry we get
\begin{eqnarray*}
\psi(x_\lambda)&=&\frac{1}{2\pi}\int_{H_\lambda}\log|x_\lambda-y|\, H(\varphi(y))\,dy+\frac{1}{2\pi}\int_{H_\lambda}\log|x_\lambda-y_\lambda|\, H(\varphi(y_\lambda))\,dy\\
&=&\frac{1}{2\pi}\int_{H_\lambda}\log|x-y_\lambda|\, H(\varphi(y))\,dy+\frac{1}{2\pi}\int_{H_\lambda}\log|x-y|\, H(\varphi(y_\lambda))\,dy.
\end{eqnarray*}
This implies according to the definition of $\varphi_\lambda$
\begin{eqnarray*}
\varphi_\lambda(x)&=&-2\lambda\Omega(\lambda-x_1)+\frac{1}{2\pi}\int_{H_\lambda}\log\Big(|x-y_\lambda|/|x-y|\Big)\, \Big[H(\varphi(y))-H(\varphi(y_\lambda))\Big]dy.
\end{eqnarray*}
Since $H$ is increasing and according to the the assumption $\varphi_\lambda\geq0$ in $H_\lambda$, one gets
$$
H(\varphi(y))-H(\varphi(y_\lambda))\geq0, \quad \forall\, y\in H_\lambda.
$$
We combine this with the obvious  geometric fact: for $\lambda>0,$
$$
|x-y_\lambda|\geq |x-y|,\quad \forall\, x, y\in H_\lambda.
$$
Consequently,
$$
\varphi_\lambda(x)\geq-2\lambda\Omega(\lambda-x_1)>0,\quad \forall x\in H_\lambda.
$$
It remains to check that  $\partial_{x_1}\varphi_\lambda<0$ on $T_\lambda.$ Straightforward computations give for $x\in T_\lambda$ corresponding to  $\lambda= x_1,$
\begin{eqnarray*}
\partial_{x_1}\log\Big(|x-y_\lambda|/|x-y|\Big)&=&\frac{x_1-(2\lambda-y_1)}{|x-y_\lambda|^2}-\frac{x_1-y_1}{|x-y|^2}\\
&=&2\frac{y_1-\lambda}{|x-y|^2}
\end{eqnarray*}
which implies
$$
\partial_{x_1}\log\Big(|x-y_\lambda|/|x-y|\Big)<0,\quad \forall y\in H_\lambda.
$$
Therefore we get for $x\in T_\lambda$
\begin{eqnarray*}
\partial_{x_1}\varphi_\lambda(x)&=&2\lambda\Omega+\frac{1}{2\pi}\int_{H_\lambda}\partial_{x_1}\,\log\Big(|x-y_\lambda|/|x-y|\Big)\, \Big[H(\varphi(y))-H(\varphi(y_\lambda))\Big]dy\\
&<&0.
\end{eqnarray*}
This concludes the proof of Proposition \ref{prop3}.
\end{proof}
Now we discuss a more precise  statement of Proposition \ref{prod1}.
\begin{prop}\label{prop4}
The following claims hold true.
\begin{enumerate}
\item For any $\lambda>0$, we have
$$
\varphi_\lambda(x)>0,\quad\forall x\in H_\lambda.
$$
\item The function $\varphi$ is radial and satisfies for any $r>0$
$$
\partial_r\varphi(r)<0.
$$
\end{enumerate}

\end{prop}
\begin{proof}
{\bf{(1)}} Define the set
$$
I\triangleq \big\{\lambda>0, \varphi_\lambda>0\quad\hbox{in} \quad H_\lambda\big\}.
$$
According to Proposition \ref{prop2} this set is nonempty and contains the interval $[\lambda^\star,\infty[.$ Let $(\alpha,\infty[$ be the largest interval contained in $I$ and  assume by contradiction  that $\alpha>0.$ First observe that by a continuation principle  $\varphi_\alpha\geq0$ in $H_\alpha$ and therefore  Proposition \ref{prop3} implies that,
$$ \forall x\in H_\alpha,\quad \varphi_\alpha(x)>0.
$$
This means  that $\alpha$ belongs to this maximal interval and thus it coincides with the closed interval $[\alpha,\infty[.$ By the maximality of the interval $[\alpha,+\infty[$, there exist two sequences $(\alpha_n)$ and $(x_n)$ with the following properties
$$
0<\alpha_n<\alpha,\,\lim_{n\to\infty}\alpha_n=\alpha;\quad  x_n \in H_{\alpha_n}\quad \hbox{and}\quad \varphi_{\alpha_n}(x_n)\le0.
$$  According to Proposition \ref{prop2} the sequence $(x_n)$ is bounded and therefore up to an extraction we can assume that this sequence converges to some point $\overline{x}\in H_{\alpha}\cup T_{\alpha}$. By passing to the limit using the continuity of the map $(\alpha,x)\mapsto\varphi_\alpha(x)$ it is easy to check that
$$
\varphi_{\alpha}(\overline{x})\le0
$$
and consequently $\overline{x}\in T_\alpha$ because $\varphi_\alpha(x)>0$ in $H_\alpha$. Using Proposition \ref{prop3} we \mbox{get  $\partial_{x_1}\varphi_\alpha(\overline{x})<0.$} This contradicts the assumption $\varphi_{\alpha_n}(x_n)\le 0$ because  it yields
$$
\frac{\varphi(x_n^1,x_n^2)-\varphi(2\alpha_n-x_n^1,x_n^2)}{2(\alpha_n- x_n^1)}\leq0.
$$ 
By passing to the limit as $n$ goes to $\infty$ we find $\partial_{x_1}\varphi(\overline{x})\geq0$, which is impossible.

\vspace{0,2cm}

{$\bf{(2)}$} By passing to the limit as $\lambda\to0$ in the first point of the proposition and using the continuity of $\varphi$ one gets
$$
\varphi(-x_1,x_2)\geq \varphi(x_1,x_2),\quad\forall x_1\geq0. 
$$
By changing the orientation we get the reverse  inequality and therefore  we get the equality $
\varphi(-x_1,x_2)=\varphi(x_1,x_2) 
$. This can be implemented for  any moving plane (line) and consequently the solution $\varphi$ will be invariant by reflection  with respect to any line passing by the origin. This means that $\varphi$ is radial. Now we can use  Proposition \ref{prop3} by taking  $\lambda=x_1>0$, leading to 
$$
\partial_{x_1}\varphi(x_1,0)<0.
$$ 
Since $\varphi$ is radial we get $\partial_r\varphi(r)<0$ for any $r>0$ and therefore the proof of Proposition \ref{prop4} is now complete.
\end{proof}
%%
%%%\OMega
\section{Case $\Omega=\frac12$}
In this section we shall be concerned with the proof of Theorem \ref{thm2} dealing with the  special angular velocity $\Omega=\frac12.$ We will prove the non existence of non trivial V-states rotating with this angular velocity. 
%\begin{Theo}
%Let $\chi_D$ be a simply  connected rotating patch with the angular velocity $\Omega=\frac12$ then it must be a disc.
%\end{Theo}
\begin{proof}
Recall that the boundary of any V-state is described by the equation
$$
\varphi\triangleq\mu+\frac14|x|^2-\psi(x)=0,\forall x\in \partial D.
$$ 
It is easy to see that
$$
\Delta \varphi=1-\chi_D, \,\, x\in \RR^2
$$
and therefore 
\begin{equation}\label{B}
\left\{ \begin{array}{ll}
\Delta \varphi=0\quad \hbox{in}\quad D&\\
\varphi=0,\quad \partial D.
\end{array} \right. 
\end{equation} 
By the uniqueness property of the Cauchy problem we obtain 
$$
\forall x\in D, \quad \varphi(x)=0.
$$
This implies that
$$
\psi(x)=\mu+\frac14|x|^2,\quad \forall x\in \overline{D}
$$
and consequently,
\begin{eqnarray*}
4\partial_z\psi(z)=\overline{z},\quad \forall z\in {D}.
\end{eqnarray*}
By the extension principle for continuous functions we get
\begin{eqnarray*}
\mathcal{C}(\chi_D)(z)&=&-\overline{z},\quad \forall z\in \overline{D}.
\end{eqnarray*}
%For the proof we shall use the equation \eqref{pilot1} which is equivalent  to
%$$
%\operatorname{Re} F(z)=Cte,\quad z\in \partial D
%$$
%where 
%$$
%F(z)\triangleq\frac{1}{2i\pi}\int_{\partial D}\overline{\xi}\log\big(1-\frac{z}{\xi}\big)d\xi,\quad z\in D.
%$$
%The function $F$ is holomorphic inside the open domain $D$ and has a continuous extension to the boundary $\partial D$. Therefore $\operatorname{Re} F$ is harmonic in $D$ and constant on the boundary, which implies by the maximum principle that $\operatorname{Re} F$ is constant in $D$. Using  the Cauchy-Riemann equations we obtain that  $F$ is  constant in $D$. By differentiation we obtain
%$$
%F^\prime(z)=\frac{1}{2i\pi}\int_{\partial D}\frac{\overline{\xi}}{\xi-z}d\xi=0, z\in D.
%$$
%From the Cauchy-Pompeiu  formula one gets
%$$
%\overline{z}=\frac{1}{2i\pi}\int_{\partial D}\frac{\overline{\xi}}{\xi-z}d\xi-\mathcal{C}(\chi_D)(z),\quad z\in D
%$$
with
$$
\mathcal{C}(\chi_D)(z)\triangleq\frac{1}{\pi}\int_{D}\frac{1}{\xi-z}dA(\xi)
$$
being the Cauchy transform of the domain $D.$ We claim that necessary the domain is a disc. This was proved in a general framework  \mbox{in \cite{HMV1}} and for the convenience of the reader we shall give a proof in the few next lines.
Set 
$$
 G(z)\triangleq z\,\mathcal{C}(\chi_D)(z)\quad\hbox{and}\quad H\triangleq\operatorname{Im} G(z),
$$ defined in $\CC\backslash \overline{D}$. The function $G$ is analytic in this domain and has a continuous extension up to the boundary of $\widehat\CC\backslash D$. Note that $\displaystyle{\lim_{z\to\infty}G(z)\in \RR}$ and consequently the harmonic function $H$ vanishes on the boundary  $\partial D\cup\{\infty\}$. By the  maximum principle this implies that  $H$  vanishes everywhere in $\CC\backslash D$. Therefore $G$ is constant everywhere and in particular in $\partial D$, that is,
$$
|z|^2=Cte, \quad\forall \, z\in \partial D.
$$
This proves that $\partial D$ is a disc.

% Let {$\displaystyle{\phi(\xi)=\xi+\sum_{n\geq 0}\frac{a_n}{\xi^n}}$} be conformal. Then 
%$$
%\frac{\xi \phi^\prime(\xi)}{\phi(\xi)-w}=\sum_{n\geq0}{F_n(w)}\xi^{-n}
%$$
%with ${F_n}$ a monic polynomial of degree $n$.
% The boundary of the V-states satisfies
%$$
%{\mathcal{G}(\Omega,\phi)}={(\Omega-\frac12)}\, |{\phi(w)}|^2+{\textnormal{Re}}\sum_{n\geq0}\overline{a_n} {G_n}(\phi(w))=Cte
%$$
%with ${G_n}(z)=\int_0^z{F_n}(\xi)d\xi$
%
%$\bullet$ Case $\Omega=\frac12.$
%We assume that ${\partial D}$ belongs to ${C^{1+\varepsilon}},$ then 
%
%$$
%{\textnormal{Re}}\sum_{n\geq0}\overline{a_n} {G_n}(z)=Cte, \quad z\in \partial D
%$$
% The maximum principle with the holomorphy property of the series  give: 
% $$
%\sum_{n\geq0}\overline{a_n} {G_n}(z)=Cte, \quad z\in D
% $$
% By differentiating and taking the values at the boundary we obtain
%$$
%\sum_{n\geq0}\overline{a_n} {F_n}(\varphi(w))=0, \quad w\in \mathbb{T}.
%$$
% We have the identity
%$$
%F_n(\varphi(w))=w^n+\sum_{k\geq1} {{\beta_{nk}}}w^{-k}
%$$ 
%with $\beta_{nk}$ the grunsky coefficients. 
%  This gives  
%  $$
%  \sum_{n\geq0}\overline{a_n} w^n+  \sum_{n\geq0}\overline{a_n}\sum_{k\geq1} {{\beta_{nk}}}w^{-k}=0, w\in \mathbb{T}.
%  $$
%  Therefore we get 
%  $$
%  a_n=0,\,\forall n\in \mathbb{N}
%  $$
%  and consequently $\phi(\xi )= \xi$ meaning that $D$ is a disc.

\end{proof}
\begin{gracies}
We would like to thank Joan Verdera for the fruitful discussion about the stationary patches and for pointing out the valuable references \cite{Fran} and \cite{wolfgang}. This work was
partially supported by the ANR
project Dyficolti ANR-13-BS01-0003-01.
\end{gracies}

\end{document}